\documentclass[12pt]{amsart}

\usepackage[dvips]{graphics}
\usepackage{epsfig}
\usepackage{amssymb}
\usepackage{amsthm}
\usepackage[mathscr]{eucal}
\usepackage{enumitem}
 \usepackage[usenames,dvipsnames]{pstricks}
 \usepackage{epsfig}

\usepackage{color}

\newtheorem{Theorem}{Theorem}[section]
\newtheorem{theorem}[Theorem]{Theorem}

\newtheorem{proposition}[Theorem]{Proposition}

\newtheorem{corollary}[Theorem]{Corollary}
\newtheorem{Lemma}[Theorem]{Lemma}
\newtheorem{lemma}[Theorem]{Lemma}

\newtheorem{fact}[Theorem]{Fact}

\theoremstyle{definition}

\newtheorem{example}[Theorem]{Example}

\newtheorem{remark}[Theorem]{Remark}
\newtheorem{remark/def}[Theorem]{Remark/Definition}

\newtheorem{definition}[Theorem]{Definition}

\newtheorem{notation}[Theorem]{Notation}
\newtheorem*{notationn}{Notation}

\newsavebox{\indbin}
\savebox{\indbin}{\begin{picture}(0,0)
\newlength{\gnu}
\settowidth{\gnu}{$\smile$} \setlength{\unitlength}{.5\gnu}
\put(-1,-.65){$\smile$} \put(-.25,.1){$|$}
\end{picture}}
\def \indo {\mathop{\smile \hskip -0.9em ^| \ }}

\newcommand{\be}{\begin{enumerate}}
\newcommand{\bd}{\begin{defn}}

\newcommand{\bt}{\begin{theorem}}
\newcommand{\bl}{\begin{lemma}}
\newcommand{\ee}{\end{enumerate}}
\newcommand{\ed}{\end{defn}}
\newcommand{\et}{\end{theorem}}
\newcommand{\el}{\end{lemma}}

\newcommand{\la}{\langle}
\newcommand{\ra}{\rangle}

\newcommand{\CP}{{\mathcal P}}

\newcommand{\CA}{{\mathcal A}}

\newcommand{\CC}{{\mathcal C}}

\newcommand{\CM}{{\mathcal M}}

\newcommand{\dom}{\mbox{dom}}

\newcommand{\sdist}{\operatorname{Sd\:}}
\newcommand{\bsdist}{\operatorname{\widehat{S}d\:}}

\newcommand{\Th}{\operatorname{Th}}

\def\eq{\operatorname{eq}}
\def\heq{\operatorname{heq}}

\def\dcl{\operatorname{dcl}}

\def\dom{\operatorname{dom}}

\def\acl{\operatorname{acl}}

\def\ker{\operatorname{Ker}}

\def\Im{\operatorname{Im}}

\def\tp{\operatorname{tp}}

\def\ltp{\operatorname{Ltp}}

\def\Sym{\operatorname{Sym}}

\def\raw{\rightarrow}

\def\ob{\operatorname{Ob}}
\def\mor{\operatorname{Mor}}
\def\supp{\operatorname{supp}}
\def\Bd{\partial}

\def\sign{\operatorname{sign}}

\title{A classification of 2-chains having 1-shell boundaries in rosy theories}
\author{Byunghan Kim}
\author{SunYoung Kim}
\author{Junguk Lee}

\address{Department of Mathematics\\ Yonsei University\\
50 Yonsei-Ro, Seodaemun-Gu\\
Seoul 120-749, Korea}

\email{bkim@yonsei.ac.kr}
\email{sy831@yonsei.ac.kr}
\email{ljw@yonsei.ac.kr}

\thanks{The  first author was supported by NRF of Korea grants 2011-0021916 and 2013R1A1A2073702.
The second and third authors were supported by NRF of Korea grants 2012-030479 and 2012-044239.
The third author was supported by European Community's Seventh Framework Programme [FP7/2007-2013] grant 238381.}

\begin{document}

\begin{abstract}
We classify, in a non-trivial amenable collection of functors, all 2-chains up to the relation of having the same  1-shell  boundary. In particular, we prove that in a rosy theory, every 1-shell of a Lascar strong type
is the boundary of some 2-chain, hence making  the 1st homology group
trivial.

We also show that, unlike in simple theories, in rosy theories there
is no upper bound on the minimal lengths of $2$-chains whose boundary
is a $1$-shell.
\end{abstract}

\maketitle

\section{Introduction}

In \cite{GKK2},\cite{GKK3}, J. Goodrick, A. Kolesnikov and the first author developed a homology theory for  any amenable collection of
functors in a very general context. But the most interesting examples
appear in model theory.  Namely, given any strong type $p\in S(A)$ in
a rosy theory $T$, we may assign a non-trivial amenable collection of
functors preserving thorn-independence and compute the corresponding
homology groups. By the general theory, if $T$ has $n$-complete
amalgamation ($n\geq 2$) over $A=\acl(A)$  then the $(n-1)$-th
homology group of $p\in S(A)$ consists of $(n-1)$-shells with the
support $n+1= \{0, \ldots, n \}$. Hence, in any simple $T$ (where, due
to $3$-amalgamation, every $1$-shell is the boundnary of some
$2$-simplex), the 1st homology group is trivial. But the question
remained whether the same would  hold in rosy theories. In this paper,
we show that the answer  is \emph{yes} (as long as $p$ is a Lascar
type). A crucial ingredient in our proof is the fact that $a$ and $b$
realize the same Lascar type if and only if their Lascar distance  is
finite, i.e., $d_A(a, b) <\omega$. In the proof, the number of 2-simplices involved in a 2-chain having the 1-shell boundary is proportional to
$d_A(a,b)$. Therefore one may guess that, there does not exist a uniform bound
for the minimal lengths of 2-chains with 1-shell boundaries for various Lascar types in rosy theories, contrary to
the case of simple theories where the bound is 1, due to 3-amalgamation. A series of rosy examples in \cite{CLPZ}  where the Lascar distances increase are candidates.
However  in order to confirm that in each example that a candidate 2-chain
has the expected  minimal length, we need to rule out all other possibilities. For this goal we start to classify all the 2-chains having the same 1-shell boundary in a very general amenable context. The classification also has its own research interests. We obtain some interesting and important results in regard to the classification.

There are basically two operations on the class of 2-chains preserving the length  and boundary of a chain. The first one is called the crossing (CR-)operation and the second one is called the renaming-of-support (RS-)operation. Two 2-chains are said to be equivalent if one is obtained from the other by finitely many applications of the two operations. 


In the remainder of this section,
we recall the definitions of an amenable family of functors and the corresponding homology groups
introduced in \cite{GKK2},\cite{GKK3}. We thank Hyeung-Joon Kim and John Goodrick for their valuable suggestions and  comments.
\begin{notationn} 
Throughout the paper, $s$ denotes an arbitrary finite set of natural numbers. Given any subset $X\subseteq \CP(s)$, we may view $X$ as a category where for any $u, v \in X$, $\mor(u,v)$ consists of a single morphism $\iota_{u,v}$ if $u\subseteq v$, and $\mor(u,v)=\emptyset$ otherwise. If $f\colon X \raw \CC$ is any functor into some category $\CC$ then for any $u, v\in X$ with $u \subseteq v$, we let $f^u_v$ denote the morphism $f(\iota_{u,v})\in \mor_{\CC}(f(u),f(v))$. 
We shall call $X\subseteq \CP(s)$ a \emph{primitive category} if $X$ is non-empty and  \emph{downward closed}, i.e.,  for any $u, v\in \CP(s)$,  if $u\subseteq v$ and $v\in X$ then $u\in X$. (Note that all primitive categories have  the empty set $\emptyset\subset \omega$ as an object.)
\end{notationn}

\begin{remark/def}
Given any primitive categories $X$ and $Y$, define
\[ X+Y: = \{ t\cup k \mid t\in X,\, k\in Y \}
\]
which is clearly a primitive category itself containing $X$ and $Y$ as subcategories. And, for any $t\in X$, define
\[ X_t:= \{ k\in X \mid t\cap k = \emptyset \}\quad \mbox{and} \quad X|_t:=\{ k\in X_t\mid t\cup k \in X \}
\]
both of which are clearly primitive subcategories of $X$. Observe:
\be
\item $X|_t \subseteq X_t \subseteq X$
\item $X\subseteq X_t + \CP(t)$
\item $X|_t = \bigcup \{ \CP(u\setminus t) \mid t\subseteq u \in X \}$.
\ee

\smallskip

\noindent Moreover, it is easy to check that the following are
equivalent:\\
 $X = X_t + \CP(t)$ $\iff$ $X_t = X|_t$ $\iff$ $X = \bigcup \{ \CP(u)\mid t\subseteq u \in X \}$.

\smallskip

\noindent 
If one of these equivalent conditions holds, we shall say that  $X$  \emph{splits at $t$}.

\smallskip

\noindent For any functor $f\colon X \raw \CC$ to some category $\CC$ and for any $t\in X$, the \emph{localization of $f$ at $t$} is the functor $f|_t\colon X|_t \raw \CC$ defined as follows: for any $u\subseteq v\in X|_t$,  $(f|_t)^u_v  = f^{u\cup t}_{v\cup t}$ and $f|_t (v) = f(t\cup v)$.
\end{remark/def}

\begin{definition}\label{isoper}
Let $X\subseteq \CP(s)$ and $Y\subseteq \CP(t)$ be any primitive categories (where $s, t$ are some finite sets of natural numbers). And let  $f\colon X \raw \CC$ and $g\colon Y \raw \CC$ be any functors to some category $\CC$.
\be
	\item  We say that $f$ and $g$ are {\em isomorphic} if there is  an order-preserving bijection $\tau \colon s \raw t$ such that  $Y=\{ \tau(u)\mid u\in X \}$ and there is a family of isomorphisms $\{h_u \colon f(u)\rightarrow g(\tau(u))\mid u\in X\} \subseteq \mor(\CC)$ such that, for any $u\subseteq v\in X$,
        \begin{center}
            $h_v \circ f^u_v =
            g^{\tau(u)}_{\tau(v)} \circ h_u$.
        \end{center}
 \item \label{loc.per} We say that $f$ and $g$ are {\em permutations of each other} if there is a bijection $\sigma \colon s \raw t$ (not necessarily order-preserving) such that $Y=\{\sigma(u) \mid  u \in X\}$ and, for any  $u\subseteq v\in Y$, $g(v)=f(\sigma^{-1}(v))$ and $(g)^u_v=f^{\sigma^{-1}(u)}_{\sigma^{-1}(v)}$. In this case, we write $g=f\circ\sigma^{-1}$.
\ee
\end{definition}

Note that, if $f$ and $g$ are permutations of each other via an \emph{order-preserving} map $\sigma\colon s \raw t$, then $f$ and $g$ are isomorphic.

\begin{definition}
Let $\CA$ be a non-empty family of functors from various primitive categories into some fixed category $\CC$. We say that $\CA$ is \emph{amenable}  if it satisfies the following properties:
\be
\item (Closed under isomorphism and permutation)  If  $f\in \CA$ then any functor $g$ which is isomorphic to $f$ or is a permutation of $f$ also belongs to $\CA$.
\item (Closed under restriction and union) For any functor $f\colon X\raw \CC$ from some primitive category $X$ into $\CC$,
\[ f\in \CA \ \Leftrightarrow\ \mbox{for every $t\in X$}, f\upharpoonright\CP(t)\in \CA.
\]
\item (Closed under localization) If $f\colon X \raw \CC$ is any functor in $\CA$ then for every $t\in X$, $f|_t \colon X|_t \raw \CC$ is also in $\CA$.
\item (Extensions of localizations are  localizations of extensions) Let  $f\colon X \raw \CC$ be any functor in $\CA$ which splits at some $t\in X$. Then whenever $f|_t$ can be extended to some functor $g\colon Z \raw \CC$ in $\CA$ where $t\cap \bigcup Z = \emptyset$, $f$ can be extended to some functor $h\colon \CP(t) +Z \raw \CC$ in $\CA$ such that $h|_t = g$.
\ee
\end{definition}

\begin{definition}
By a ({\em regular}) $n$-{\em simplex} in a category $\CC$, we mean a functor $f\colon \CP(s) \raw \CC$ where $s\subseteq \omega$ has the size $n+1$. We call $s$ the {\em support} of $f$ and denote it by $\supp(f)$.
\end{definition}

\begin{definition}
Let $\CA$ be an amenable family of functors into some category $\CC$. Let $B\in \ob(\CC)$. If $f$ is a functor in $\CA$ such that $f(\emptyset) = B$,  we shall say that $f$ is \emph{over $B$}. And we define:
\begin{align*}
&S_n(\CA; B):= \{\, f\in \CA \mid \mbox{$f$ is a regular $n$-simplex  over $B$}\, \}\\
&C_n(\CA; B): =  \mbox{the free abelian group generated by $S_n(\CA; B)$.}
\end{align*}
The elements of $C_n(\CA; B)$ are called the \emph{$n$-chains over $B$ in $\CA$}. For each $i=0, \ldots,  n$, we  define a group homomorphism
\[ \Bd^i_n\colon C_n(\CA; B) \raw C_{n-1}(\CA; B)
\]
by  letting, for any $n$-simplex   $f\colon \CP(s) \raw \CC$ in $S_n(\CA; B)$ where  $s = \{ s_0 <\cdots < s_n \}$, 
\[ \Bd_n^i(f):= f\upharpoonright \CP(s\setminus\{s_i\})
\]
and then extending linearly to all $n$-chains in $C_n(\CA; B)$. Then we define  the \emph{boundary map}
\[ \Bd_n\colon C_n(\CA; B)\raw C_{n-1}(\CA; B)
\]
by
\[ \Bd_n(c):=\sum_{0\leq i\leq n}(-1)^i \Bd_n^i(c).
\]
We shall often refer to $\Bd_n(c)$ as the \emph{boundary of $c$}. Next, we define:
\begin{align*}
&Z_n(\CA; B):= \ker \, \Bd_n\\
& B_n(\CA; B):= \Im \, \Bd_{n+1}.
\end{align*}

\noindent The elements of $Z_n(\CA; B)$ and $B_n(\CA; B)$ are called \emph{$n$-cycles} and \emph{$n$-boundaries}, respectively. It is straightforward to check
\[ \Bd_{n-1}\circ \Bd_{n} =0.
\]
Hence we may define
\[ H_n(\CA; B):= Z_n(\CA; B)/B_n(\CA; B)
\]
called the \emph{$n$}-th \emph{(simplicial) homology group of $\CA$ over $B$}.
\end{definition}

\begin{notation}
\be
\item 
For $c\in Z_n(\CA; B)$, $[c]$ denotes  the coset of $B_n(\CA; B)$ containing $c$.
\item When $n$ is clear from context, we shall often omit $n$ from $\Bd^i_n$ and $\Bd_n$, writing simply as  $\Bd^i$ and $\Bd$.
\item When we write an $n$-chain $c\in C_n(\CA; B)$ as
\[c = \sum_{i=1}^kn_if_i
\]
we shall assume, unless stated otherwise, that $n_i$'s are nonzero
integers and $f_i$'s are distinct $n$-simplices. (This form is called
the \emph{standard form} of a chain.) For such an $n$-chain $c$, we
define the \emph{length} of $c$ and the \emph{support} of $c$ as $|c| : = \sum_{i=1}^k|n_i|$ and
$\supp(c):= \bigcup_{i=1}^k \{\supp(f_i) \}$, respectively.

\item For $c, d\in C_n(\CA; B)$, we say that $d$ is a \emph{subchain (or subsummand) of $c$} if they are in the standard forms
\[c = \sum_{i=1}^kn_if_i\qquad \mbox{and}\qquad d = \sum_{i\in
  J}m_if_i,\]
where $J\subseteq \{ 1, \ldots, k \}$ and, for each $i\in J$,
$n_i\cdot m_i>0$ and $ |m_i|\leq |n_i|$.
\ee
\end{notation}


\begin{remark/def}\label{subsum} 
Let $c$ be any $n$-chain and let $d$ be a subsummand of $c$. For any $n$-chain $d'$,  we shall say that the $n$-chain
\[ c':= c - d + d'
\]
is  obtained  by \emph{replacing the subsummand $d$ in $c$ by $d'$}. Note that, if $|d'|\leq |d|$ then  $|c'|\leq |c|$.
\end{remark/def}

\begin{remark/def}\label{sigma*}
Given any bijection $\sigma \colon \omega \raw \omega$ (not
necessarily order-preserving), we may induce an automorphism
$\sigma^*_n \colon C_n(\CA; B) \raw C_{n}(\CA; B)$
 for each $n$ as follows: for any $n$-chain $c=\sum_i n_if_i\in
 C_n(\CA,B)$, where each $f_i$ is an $n$-simplex with
$s_i:=\supp(f_i)=\{s_{i,0}<\cdots<s_{i,n}\}$, we let
$\sigma_i:=\sigma\restriction s_i$ and
$t_i:=\sigma_i(s_i)=\{t_{i,0}<\cdots<t_{i,n} \}$. 
 We define $$\sigma^*(c):=\sum_i n_i|\sigma_i|f_i\circ \sigma_i^{-1}$$
 (see Definition \ref{isoper}(\ref{loc.per})) with  $|\sigma_i|:=\sign
 (\sigma'_i)$ ($=\:\pm 1$) where
 $\sigma'_i\in \Sym(n+1)$ such that  for $j\leq n$, $\sigma_i(s_{i,j})=t_{i,\sigma'_i(j)}$.
For example $$\sigma^*(f_i)={|\sigma_i|}f_i\circ \sigma_i^{-1}.$$
Moreover,  $\sigma^*$ commutes with the boundary map, i.e., $\Bd\circ \sigma^* = \sigma^*\circ \Bd$.  This can be verified inductively by first checking the case where $\sigma$ is a transposition.
\end{remark/def}

Next we define the amalgamation properties. For $n=\{0,\ldots,n-1\}$, we let $\CP^-(n):=\CP(n)\setminus \{n\}$. i.e., $\CP^-(n)$ is the set of all the \emph{proper} subsets of $n$.

\begin{definition}
Let $\CA$ be an amenable family of functors into a category $\CC$.
\be
\item $\CA$ has \emph{$n$-amalgamation} ($n\geq 1$) if every functor $f\colon \CP^-(n) \raw \CC$ in $\CA$ can be extended to some functor $g\colon \CP(n)\raw \CC$ in $\CA$.
\item $\CA$ has \emph{$n$-complete amalgamation} (written $n$-CA) if it has $k$-amalgamation for every $1\leq k\leq n$.
\item $\CA$ has \emph{strong $2$-amalgamation} if, whenever $f\colon \CP(s) \raw \CC$ and $g\colon \CP(t)\raw \CC$ are simplices in  $\CA$ which agree on $\CP(s\cap t)$, then there exists some simplex $h\colon \CP(s\cup t) \raw \CC$ in $\CA$ extending both $f$ and $g$.
\ee
\end{definition}

\begin{remark}
It is easy to verify that, for any amenable family $\CA$:
\be
\item strong $2$-amalgamation $\, \Rightarrow\, 2$-amalgamation.
\item ($1$-amalgamation  $+$ strong $2$-amalgamation) $\, \Rightarrow\, \CA$ has  $n$-simplices for every $n\geq 0$.
\ee
\end{remark}

\begin{definition}
An amenable family of functors  is called \emph{non-trivial} if it  has $1$-amalgamation and strong $2$-amalgamation (in particular, it has  $2$-CA).
\end{definition}

\begin{definition}
An $n$-chain  $c\in C_n(\CA; B)$ is called an \emph{$n$-shell} if it is in the form 
\[c =\pm \sum_{0\leq i\leq n+1}\limits (-1)^i f_i
\]
where $f_i$'s are $n$-simplices satisfying
\[ \Bd^i f_j = \Bd^{j-1} f_i \quad \mbox{whenever $0\leq i < j \leq n+1$.}
\]
We define $E_n(\CA; B):= \{\, c\in C_n(\CA; B) \mid c \ \mbox{is an $n$-shell}\ \}$.
\end{definition}

\medskip

It is straightforward to verify the  following proposition.

\begin{proposition}\label{shellremark}
\be
\item $E_n(\CA; B) \subset Z_n(\CA; B)$.
\smallskip

\item For every $f\in S_n(\CA; B)$, $\Bd_n(f) \in E_{n-1}(\CA; B)$.

\smallskip
\item If $ c= \pm \sum_{0\leq i\leq n+1}\limits (-1)^i f_i$ is any $n$-shell,
  then  $|\supp(c) | = n+2$.
Moreover, there exists a unique functor $g\colon \CP^-(\supp(c)) \raw
\CC$ in $\CA$ extending all the  $f_i$'s. More precisely,  if we let $ \supp(c) = \{ s_0 <\cdots < s_{n+1}\}$,
then $g\restriction \CP(\supp(c)\setminus \{ s_i \}) = f_i$ 
for each $i$.

\smallskip

\item $\CA$ has $(n+2)$-amalgamation if and only if for any $n$-shell $c$, there exists some $(n+1)$-simplex $d$ such $c = \pm\Bd(d)$.
\ee
\end{proposition}

\begin{definition}
An amenable family  of functors has \emph{weak $3$-amalgamation} if each $1$-shell is the boundary of some $2$-chain $c$ with $|c|\leq 3$.
\end{definition}

The following result due to  \cite{GKK2}, \cite{GKK3} illustrates the importance of the notion of shell.

\begin{fact}\label{funfact}\cite{GKK2}\cite{GKK3}
Let $\CA$ be any non-trivial amenable family of functors. If $\CA$ has $(n+1)$-CA for some $n\geq 1$, then
\[ H_n(\CA; B) = \{ [c] \mid c \in E_n(\CA; B),\ \supp(c) = \{ 0, \ldots, n+1 \} \, \}.
\]
In particular,
\be
\item $H_1(\CA; B) = 0\, \Leftrightarrow\ E_1(\CA; B)\subset B_1(\CA; B)$
\smallskip

\item If $\CA$  has weak $3$-amalgamation then $H_1(\CA; B) = 0$.
\ee
\end{fact}

\medskip

{\bf In the remainder  of the paper, $\CA$ shall denote a  non-trivial amenable family of functors into a category $\CC$.}

\medskip

Now we begin to talk about the prototypical examples of an amenable
family of functors : complete types in rosy theories.
{\bf In the sequel we work with a large saturated model $\CM=\CM^{\eq}$ and its theory $T$ which is rosy.}
 Recall that a theory is called  {\em rosy} if there is a ternary independence relation
$\indo$ on the small sets of its model, satisfying the basic
independence properties. (See \cite{A}, \cite{EO} for the  precise
definition.)  We take $\indo$ here to be thorn-independence. Any
simple or o-minimal theory is known to be rosy. Moreover, if a simple theory $T$ has elimination of hyperimaginaries then non-forking independence is equal to thorn-independence.  {\bf So we assume that any simple $T$ in this paper
  has  elimination of hyperimaginaries}. (Of course this is just for convenience as we can work in $\CM^{\heq}$ without the assumption.) In particular, we assume that 3-amalgamation holds over any algebraically closed set in simple $T$.

\smallskip

We fix any algebraically closed small subset $B\subseteq \CM$ and consider the category $\CC_B$ whose objects are all the small subsets of $\CM$ containing $B$, and whose morphisms are elementary maps over $B$ (i.e., fixing $B$ pointwise). We also fix any $p(x)\in S(B)$ (where $x$ could be an infinite tuple). When $f$ is any functor from a primitive category $X$ into $\CC_B$ and  $u\subseteq v\in X$, we shall abbreviate $f^u_v(f(u))$  as $f^u_v(u)$.

\begin{definition}
By a \emph{closed independent functor in $p(x)$}, we mean a functor $f$ from some primitive category $X$ into $\CC_B$ satisfying the following:

\be
\item Whenever $\{ i \}\subset \omega$ is an object in $X$, we can choose a realization $b\models p(x)$ such that, if we let $C:=f^{\emptyset}_{\{ i \}}(\emptyset)$ then $f(\{ i \}) = \acl(Cb)$ and $b\indo_B C$.

\smallskip

\item Whenever $u(\neq \emptyset)\subset \omega$ is an object in $X$, we have
\[ f( u) = \acl \left( \bigcup_{i\in u} f^{\{ i \}}_u(\{ i \}) \right)
\]
and $\{f^{\{i\}}_u(\{i\})|\ i\in u \}$ is independent over $f^\emptyset_u(\emptyset)$. 
\ee
\end{definition}

\smallskip

We let $\CA(p)$ be the family of all closed independent functors in $p$.

\begin{fact}\cite{GKK3}
$\CA(p)$ is a non-trivial amenable family of functors.
\end{fact}

\begin{notation}
We shall abbreviate $S_n(\CA(p);B), C_n(\CA(p); B), \ldots$ as  $S_n\CA(p), C_n\CA(p), \ldots$. We shall also abbreviate $H_n(\CA(p); B)$ simply as $H_n(p)$.  Other than this, we use standard notation. For example
$a\equiv_Ab$ denotes $\tp(a/A)=\tp(b/A)$;  and $a\equiv^L_Ab$ denotes $\ltp(a/A)=\ltp(b/A)$, i.e., the Lascar (strong) types of $a,b$ over $A$ are the same.
\end{notation}

\section{$H_1(p)$ in rosy theories}
 \label{sec:h1}

If a theory $T$ is simple then due to 3-amalgamation and Fact \ref{funfact},
we know $H_1(p)=0$. In this section we show the same holds for any rosy $T$
as long as $p$ is a Lascar type.

\smallskip

Let $f\colon X \raw \CC_B$ be any functor in $\CA(p)$ with $f(\emptyset) = B$. If  $u\in X$ with $u = \{ i_0 <\cdots <i_k \}$, we shall write $f(u)=[a_0,\ldots,a_k]$, where $a_j\models p$,
$f(u)=\acl(B,a_0\cdots a_k)$, and $\acl(a_jB)=f^{\{i_j\}}_u(\{i_j\})$. Thus $\{a_0,\ldots,a_k\}$ is independent over $B$.

\begin{theorem}\label{overmodel}
If $B$ is a model, then $\CA(p)$ has weak $3$-amalgamation over $B$ (so $H_1(p)=0$).
\end{theorem}

\begin{proof}
Let $f = a_{12}-a_{02}+a_{01}$ be any $1$-shell in $E_1\CA(p)$ where  each $a_{ij}\colon \CP(\{i,j\})\rightarrow
\CC_B$ is a 1-simplex. We want to find a 2-chain $g$ with length 3 such that $\Bd g=f$.
For this goal there is no harm in assuming that $a_{01}(\{1\})=[a]=a_{12}(\{1\})$ and
$a_{12}(\{2\})=[b]=a_{02}(\{2\})$. Let $a_{01}(\{0\}):=[c]$ and $a_{02}(\{0\}):=[c']$, and let $q$ be a coheir of $p$ over $Babcc'$. Choose any $c''\models q$. Then $c''\indo_Babcc'$ (see \cite{EO}) and  $cc''\equiv_B c'c''$.
Now let $g:=a_{123}-a_{023}+a_{013}$ where  $a_{ij3}$ are 2-simplices  having support $\{i,j,3\}$ extending
$a_{ij}$ such that $a_{123}(\{1,2,3\})=[a,b,c'']$, $a_{023}(\{0,2,3\})=[c',b,c'']$, $a_{013}(\{0,1,3\})=[c,a,c'']$. Hence we may
assume $\Bd^0(a_{023})=\Bd^0(a_{123})$ and $\Bd^0(a_{013})=\Bd^1(a_{123})$.
But $cc''\equiv_B c'c''$ implies that we may further assume $\Bd^1(a_{013})=\Bd^1(a_{023})$. Therefore $\Bd g=f$ as desired.
\end{proof}

\begin{remark}
Of course the same proof shows that weak 3-amalgamation (over a model) holds not only in $\CA(p)$ but
 more generally inside $\CM$ (with arbitrary vertices).
\end{remark}

Recall that, for any tuples $a$ and $b$, we write $d_B(a, b) \leq n$ iff there is a sequence of tuples $c_0, \ldots, c_n$ with $c_0=a$ and $c_n=b$, such that each $c_ic_{i+1}$ begins some $B$-indiscernible sequence. The smallest such $n$ (if it exists) is denoted by $d_B(a, b)$ (called the \emph{Lascar distance between $a$ and $b$}). Recall the fact that $a\equiv^L_Bb$ iff $d_B(a,b)<\omega$ in any rosy theory.

\begin{lemma}\label{laslem}
Let $I=\langle a_0,a_1,\ldots\rangle$ be any $B$-indiscernible sequence. Then for any $c_0$ there is $c\equiv_B c_0$ such that  $c\indo_B a_0a_1$ and  $ca_0\equiv_B ca_1$.
\end{lemma}

\begin{proof} Extend $I$ to $I'$ indiscernible over $B$ having a sufficiently large length. Then by the extension axiom there is $c'\equiv_B c_0$ such that
$c'\indo_B I'$. Moreover, by the pigeonhole principle, there are $a_i,a_j\in I'$ $(i<j)$ such that $c'a_i\equiv_Bc'a_j$. Now,  by $B$-indiscernibility, there is $c$ such that
$ca_0a_1\equiv_Bc'a_ia_j$. Then $c$ is the desired tuple.
\end{proof}

\begin{theorem}\label{h1=0}
Suppose that $p$ is a Lascar strong type. Then $H_1(p)=0$.
\end{theorem}

\begin{proof} For notational simplicity we let $B = \emptyset$. As in the proof of Theorem \ref{overmodel},
given any 1-shell $f=a_{12}-a_{02}+a_{01}$ in $E_1\CA(p)$ where each $a_{ij}\colon\CP(\{i,j\})\rightarrow
\CC_B$ is a 1-simplex, we want to find a 2-chain $g$  such that $\Bd g=f$.
Again there is no harm in  assuming that $a_{01}(\{1\})=[a]=a_{12}(\{1\})$ and
$a_{12}(\{2\})=[b]=a_{02}(\{2\})$. Let $a_{01}(\{0\}):=[c]$ and $a_{02}(\{0\}):=[c'].$ By extension we can further assume $\{a,b,c,c'\}$ is independent.
Now $c,c'\models p$ and let $d(c,c')=n$. So there are $c=c_0,\ldots,c_n=c'$ such that $c_ic_{i+1}$ begins an indiscernible sequence, for $i<n$.
We can further assume that $ab\indo_{cc'}c_1c_{n-1}$; so $ab\indo c_0\cdots c_n$.
Then by Lemma \ref{laslem}, there are $e_i\models p$ ($i<n$) such that $c_ic_{i+1}\indo e_i$ and $e_ic_i\equiv e_ic_{i+1}$ (*). Again by extension
we suppose $ab\indo_{c_ic_{i+1}}e_i$, so that  each of the $\{a,c_i,e_i\},\{a,c_{i+1},e_i\}$ is independent. Moreover
each $\{ a,e_{n-1},b\},\{e_{n-1},c_n,b  \}$ is independent as well\ (**).

Now there is  $g_0:=g^+_0-g^-_0$  where $g^+_0,g^-_0$ are 2-simplices with  support $\{0,1,3\}$ such that
$g^+_0(\{0,1,3\})=[c_0,a,e_0]$ and  $g^-_0(\{0,1,3\})=[c_{1},a,e_0]$; $\Bd^0 g^+_0=\Bd^0 g^-_0$;  $\Bd^1 g^+_0=\Bd^1 g^-_0$ (this is possible by (*));
and $g^+_0$ extends $a_{01}$ (i.e., $\Bd^2 g^+_0=a_{01}$). Hence $\Bd g_0=a_{01}-\Bd^2 g^-_0$.

By iteration we can find
$g_i:=g^+_i-g^-_i$  ($0< i<n-1$) where $g^+_i,g^-_i$ are 2-simplices with  support $\{0,1,3\}$ such that
$g^+_i(\{0,1,3\})=[c_i,a,e_i]$ and  $g^-_i(\{0,1,3\})=[c_{i+1},a,e_i]$; $\Bd^0 g^+=\Bd^0 g^-$;  $\Bd^1 g^+=\Bd^1 g^-$ (this again is possible by (*));
and $\Bd^2 g^+_i=\Bd^2 g^-_{i-1}$. Therefore we have
$$\Bd (g_0+\cdots+g_{n-2})=a_{01}-\Bd^2g^-_{n-2}.$$

The rest of the proof is similar to that of Theorem \ref{overmodel}. We put $g_{n-1}:=g^+_{n-1}-a_{023}+a_{123}$ where
$a_{j23}$ is a 2-simplex with  support $\{j,2,3\}$ extending $a_{j2}$ such that  $a_{023}(\{0,2,3\})=[c_n,b,e_{n-1}], a_{123}(\{1,2,3\})=[a,b,e_{n-1}]$ (see (**)).
Also $g^+_{n-1}$ is a 2-simplex with $g^+_{n-1}(\{0,1,3\})=[c_{n-1},a,e_{n-1}]$ extending  $\Bd^2g^-_{n-2}$. Moreover again by (*), we have $\Bd^1g^+_{n-1}=\Bd^1a_{023}$.
Thus it follows
$$\Bd g_{n-1}=\Bd^2g^+_{n-1}-a_{02}+a_{12}=\Bd^2g^-_{n-2}-a_{02}+a_{12}.$$
Therefore for $g:=g_0 +\cdots+ g_{n-1}$, we have $\Bd g=f$ as desired.
\end{proof}

\section{A classification of 2-chains with a 1-shell boundary}
\label{sec:clsf}

In this section, we bring our attention back to a non-trivial amenable
family  of functors $\CA$ and  classify 2-chains of $\CA$ having
1-shell boundaries.
Basically we show that any $2$-chain having a $1$-shell boundary is
equivalent to one of two types of $2$-chains, called the {\em NR-type}
and the {\em RN-type}. 





\smallskip

We start  by introducing  two operations on $2$-chains called the \emph{crossing operation} and
the \emph{renaming-of-support operation}, respectively.
For any distinct real numbers $x$ and $y$, we shall abbreviate the open
interval $(\min\{x,y\}, \max\{x,y\})$ as $[(x,y)]=[(y,x)]$.

\begin{definition}
Let $v\in C_2(\CA; B)$ be a $2$-chain and let $w:=\epsilon_1\alpha_1
+\epsilon_2 \alpha_2$ be a subsummand of $v$, where $\alpha_i$'s
are $2$-simplices with for $i=1,2$, $\epsilon_i = \pm 1$,
$\supp(\alpha_i) = \{\ell_1, \ell_2, k_i \}$ ($k_i, \ell_i$ being all
distinct numbers) such that $\alpha_1$ and $\alpha_2$ agree on the
intersection of their domains, namely $\CP(\{ \ell_1, \ell_2 \})$. Further assume that, if we let
$\gamma:=\alpha_i \restriction \CP(\{ \ell_1, \ell_2  \})$, then
$\gamma$ does not appear in $\Bd(w)$, i.e., the two  $\gamma$ terms in
$\Bd(w)$ have opposite signs and cancel each other. 

Now by strong $2$-amalgamation,
there exists some 3-simplex $\mu$ extending both $\alpha_i$. For
$i=1,2$, let $\beta_i: = \mu \restriction \CP(\{k_1, k_2, \ell_i \})$
and\\
 $ w' : =
 \begin{cases} 
   \epsilon_2\;\beta_1 +\epsilon_1\; \beta_2&\mbox{if } \epsilon_1
     \epsilon_2 = -1,\mbox{and exactly one of } k_2, \ell_1 \mbox{ belongs to
     } [(k_1, \ell_2)] \\
\epsilon_1\; \beta_1 +\epsilon_2 \;\beta_2&\mbox{otherwise.}
 \end{cases}
$

Then the operation of replacing the subsummand $w$ in $v$ by $w'$ is
called the {\em crossing operation} (or simply {\em CR-operation}).
\end{definition}


\begin{example}\label{proper}
Let $f_0, f_1, f_2, f_3$ be $2$-simplices with $\supp(f_i) = \{ 0, 1,
2, 3 \} \setminus \{ i \}$. Assume that  $f_i$ and $f_j$ agree on
their intersection, for every pair $i, j$. Consider the $2$-chain
$ c = f_0 - f_1 + f_2$. Then we can apply the CR-operation to the subsummand $f_0 - f_1$ to obtain  a new $2$-chain
\[ c' = (-f_2+f_3) + f_2  \mbox{ or simply }  f_3.
\]
\noindent 
This example illustrates in particular that a CR-operation may not be reversible. i.e., once we apply a CR-operation to a $2$-chain, we may not be able to recover the original $2$-chain  by applying more CR-operations (unless we allow $2$-chains to be written redundantly  as $f_3 - f_2 + f_2$).
\end{example}

\smallskip

Next, we define an operation on $n$-chains called the \emph{renaming-of-support operation}. 

\begin{definition}
Let $c$ be an $n$-chain in $C_n(\CA; B)$ and let $d$ be a subsummand
of $c$. Let $j\in$ $\supp(d)$ such that $j\notin \supp (\Bd_n (d))$. (In this
situation, we say that $d$ has a \emph{vanishing support}, namely $j$,
in its boundary.) Choose any $k\notin$ $\supp(c)$ and any bijection
$\sigma\colon \omega \raw \omega$ which sends $j\mapsto k$ but which
fixes the rest of the elements in $\supp(c)$. Then the operation of
replacing the subsummand $d$ in $c$ by $\sigma^*_n(d)$ is called the
\emph{renaming-of-support operation} (or simply
\emph{RS-operation}). (See Remark/Definition \ref{sigma*} to recall the definition of $\sigma^*_n$.)
\end{definition}

\begin{remark}
When we apply the CR- and RS-operation to some subsummand of an $n$-chain $c$, the resulting $n$-chain has the same boundary as $c$ (guaranteed by the fact  that $\sigma^*_n$ commutes with the boundary map $\Bd$) and has a shorter or equal length as $c$ (by Remark/Definition \ref{subsum} and the clear fact that $\sigma^*_n$ preserves the lengths of $n$-chains).
\end{remark}

\begin{remark/def} \label{properchain}
A $2$-chain $c$ is called  \emph{proper} if its length $|c|$ does not
change after  any finitely many applications of CR/RS-operations to
its subsummands. It is clear that any $2$-chain may be reduced to a
proper $2$-chain after finitely many applications of the two
operations. Any CR-operation (also RS-operation) applied to any \emph{proper} $2$-chain is in fact reversible.
  This allows us to define an equivalence relation $\sim$ among proper
  $2$-chains by:   $c\sim c'\, \Leftrightarrow \, $ $c$ can be
  obtained from $c'$ by finitely many applications of the CR/RS-operations to its subsummands. Note that $c\sim c'$ implies $\Bd (c)  = \Bd(c')$ and $|c| = |c'|$.
\end{remark/def}



We are now ready to introduce the notions of renameable type and  non-renameable type for 2-chains having  1-shell boundaries.

\begin{definition}
Let $\alpha$ be a $2$-chain having a $1$-shell boundary.
\be
\item We say $\alpha$ is of \emph{renameable type} (or simply \emph{RN-type}) if some subsummand of $\alpha$ has a vanishing support. Otherwise, $\alpha$ is said to be of \emph{non-renameable type} (or simply \emph{NR-type}).
\item $\alpha$ is called \emph{minimal} if it is proper, and for any
  proper $\alpha'$ equivalent to $\alpha$, there does not exist any
  subsummand $\beta$ of $\alpha'$ such that $\Bd(\beta)=0$.
\ee
\end{definition}

\begin{remark}
Suppose that $\alpha$ is
a  2-chain  having a 1-shell boundary.
\be\item
Note that $\alpha$ is of NR-type iff none of the CR or RS-operation is applicable to $\alpha$, i.e. nothing else is equivalent to $\alpha$ except $\alpha$ itself.
So  an NR-type chain is  minimal.

As was the case in Example \ref{proper}, an RN-type $\alpha$ can sometimes be transformed to an NR-type by CR-operations. But if $\alpha$ is proper then its RN/NR-type is preserved under equivalence.


\item We can always find some minimal $2$-chain $\alpha'$ such that
  $\Bd(\alpha)=\Bd(\alpha')$. Such an $\alpha'$ can be  obtained from $\alpha$ by finitely many applications of  CR/RS-operations and deleting subsummands having trivial boundary.

There is a 2-chain $\beta$ with $|\beta|=5$ having a 1-shell boundary such that  any subsummand of $\beta$ does not have the trivial boundary but $\beta'$ with $|\beta'|=5$ obtained from $\beta$ by the CR-operation   has a subsummand with the boundary $0$. 

\item
If $\alpha$ is minimal then  any $\alpha'$ equivalent to $\alpha$ is minimal as well (of course  $|\alpha|=|\alpha'|$
and $\Bd(\alpha)=\Bd(\alpha')$ too).
\ee
\end{remark}

\begin{notationn}
Let $f$ be any simplex. For any subset $\{ j_0, \ldots, j_k
\}\subseteq \supp(f)$, we shall abbreviate $f\restriction \CP(\{
j_0, \ldots, j_k \})$ as $f^{j_0, \cdots, j_k}$. Also, given a chain
$c=\sum_{i\in I}n_if_i$ (in its standard form), and any subset $\{ j_0, \ldots, j_k
\}\subseteq \supp(c)$, we shall write $c^{j_0,\ldots,j_k}$ to denote
the subchain $\sum_{i\in J}n_if_i$, where $J:=\{i\in I\;|\;
\supp(f_i)=\{j_0,\ldots,j_k\}\}$.
\end{notationn}

\begin{example}
Of course any 2-simplex is of NR-type.
The following is an NR-type 2-chain with length $5$: Let $\alpha=a_1 + a_2 + a_3 - a_4 - a_5$ be a 2-chain with
2-simplices $a_i$ having
 $\supp(a_i)=\{0,1,2\}$ such that;
\be \item[$\bullet$] $a_1^{1,2}, a_2^{1,2}=a_4^{1,2}, a_3^{1,2}=a_5^{1,2}$ are distinct; 
 	\item[$\bullet$]
 $a_2^{0,2}, a_1^{0,2}=a_5^{0,2}, a_3^{0,2}=a_4^{0,2}$ are distinct; 
 	\item[$\bullet$] and so are $a_3^{0,1}, a_1^{0,1}=a_4^{0,1}, a_2^{0,1}=a_5^{0,1}$.
 \ee
Then $\alpha$ is of NR-type  with a 1-shell boundary
$a_1^{1,2}-a_2^{0,2}+a_3^{0,1}$.

\end{example}

Before stating our first main theorem of the classification, we introduce a notion called  \emph{chain-walk} which will be used in our proof.

\begin{remark}
Recall that if $\alpha$ is a 2-chain with a 1-shell boundary, then its length is always an odd positive number.
\end{remark}

For the rest of this section, {\bf we fix a  1-shell boundary $f_{12} -f_{02} +
f_{01}$ with $\supp(f_{jk})=\{j<k\}$}.

\begin{definition}
\label{chainwalk}
Let $\alpha$ be a 2-chain having the boundary $f_{12} -f_{02} +
f_{01}$.
A subchain $\beta=\sum_{i=0}^m\limits \epsilon_i b_i$ of $\alpha$ (where $\epsilon_i = \pm 1$ and   $b_i$ is a $2$-simplex, for each $i$) is called
a {\em chain-walk in $\alpha$ from $f_{01}$ to $-f_{02}$} if
        \be\item there are non-zero numbers $k_0,\ldots,k_{m+1}$ (not necessarily distinct) such that   $k_0=1$, $k_{m+1}=2$, and for $ i\leq m$, $\supp(b_i)=\{k_i,k_{i+1},0\}$;
        \item  $(\Bd \epsilon_0 b_0)^{0,1} = f_{01}$, $(\Bd \epsilon_m b_m)^{0,2}=- f_{02}$; and
        \item for $0 \le i < m$,
        $$(\Bd \epsilon_i b_i)^{0, k_{i+1}}+ (\Bd \epsilon_{i+1} b_{i+1})^{0,k_{i+1}}=0.$$
        \ee
 The sum $\sum_{i=0}^m\limits \epsilon_i b_i$ with {\em its order} is called a {\em representation} of the chain-walk $\beta$. 
 Unless said otherwise a chain-walk is written in the form of a representation.
Notice that 
a chain-walk may have more than one representation. For
example, a reordering of terms in $\beta$ above may also satisfy conditions (1)-(3). 
By a \emph{section} of the chain-walk $\beta$, we shall mean a subchain of $\beta$ in the form
\[ \beta':=\sum_{i=j}^{m'}\limits \epsilon_i b_i\quad \mbox{for some $0\leq j<m'\leq m$}
\]
and the sequence $\la k_j,k_{j+1},\ldots,k_{m'},k_{m'+1}\ra$  is called the {\em walk sequence} of $\beta'$.
 A chain-walk $\beta$ in $\alpha$ is called {\em maximal} (in $\alpha$) if it has the maximal possible length.
We say $\alpha$ is {\em centered at} $0$ if some (hence every) maximal chain-walk in $\alpha$ from  $f_{01}$ to $-f_{02}$ is, as a chain, equal to $\alpha$.

We similarly define such notions as  {\em a chain-walk in $\alpha$ from  $-f_{02}$ to $f_{12}$, $\alpha$ is centered at $2$}, and so on.
\end{definition}


\begin{remark}\label{existchainwalk}
In the definition above, if $\beta$ is a chain-walk in $\alpha$ from
$f_{01}$ to $-f_{02}$, then $0\in\supp(b_i)$ for all $i$, but $0\notin \supp(\Bd\beta -f_{01}+f_{02})$; and the walk sequence of $\beta$ is a sequential arrangement of $(\supp(b_i)\setminus \{0\})$'s without repetition of the overlapped support.

Note now that given any 2-chain $\alpha$ as in the definition above, since there are only finitely many 2-simplex terms in $\alpha$, we can always find a chain-walk say, from
$f_{01}$ to $-f_{02}$: We start with any 2-simplex term $b$ in $\alpha$ such that $\Bd^2 b=f_{01}$ and then keep finding a term in $\alpha$ (with the coefficient) cancelling out adjacent 1-simplex boundaries. This process must stop with a term having $-f_{02}$ as its boundary.

Even if $0$ is in the support of  every simplex term of  $\alpha$, it need not be centered at $0$: Let $\alpha= c_0+c_1-c_2$ such that $\Bd c_0=g_{12}-f_{02}+f_{01}$; $\Bd c_1=f_{12}-g_{02}+g_{01}$; and  $\Bd (-c_2)=-g_{12}+g_{02}-g_{01}$, where $f_{ij}\ne g_{ij}$.
Then $c_0$ itself is  a maximal chain-walk in $\alpha$ from $f_{01}$
to $-f_{02}$. Note that $\alpha= c_0+c_1-c_2$ is not a chain-walk from
$f_{01}$ to $-f_{02}$, whereas it is a chain-walk from $f_{12}$ to
$f_{01}$, i.e, $\alpha$ is centered at $1$.
\end{remark}

\begin{Lemma}\label{section}
Let $\alpha$ be a $2$-chain with the $1$-shell
boundary $f_{12} -f_{02} + f_{01}$.
Let $\beta=\sum_{i=0}^m\limits \epsilon_i b_i$  be
a chain-walk in $\alpha$, say from $-f_{02}$ to $f_{12}$. Assume there is  a section $\beta'=\sum_{i=j}^{m'}\limits \epsilon_i b_i$ of $\beta$ such that for $\supp(b_i)=\{2,k_i,k_{i+1}\}$, either $k_i\ne k_{m'+1}$ for all $i=j,\ldots,m'$;
or $k_i\ne k_{j}$ for all $i=j+1,\ldots,m'+1$. Then by finitely many applications of CR-operations to $\beta'$, we obtain a simplex $c$ with $\supp(c)=\{2,k_j,k_{m'+1}\}$ such that, for some $\epsilon = \pm 1$, 
$\beta'':=\sum_{i=0}^{j-1}\limits \epsilon_i b_i+\epsilon c
+\sum_{i>m'}^m\limits \epsilon_i b_i$ is still a chain-walk from $-f_{02}$ to $f_{12}$.
\end{Lemma}

\begin{proof}
When $j=m'$, there is nothing to prove. Assume the lemma holds when $m'-j=n$. Let us show that the lemma holds when $m'-j=n+1$.
Assume $k_i\ne k_{m'+1}$ for all $i=j,\ldots,m'$. Then we can apply the CR-operation to $\epsilon_{m'-1}b_{m'-1}+\epsilon_{m'}b_{m'}$, and we get $\epsilon'_{m'-1}b'_{m'-1}+\epsilon'_{m'}b'_{m'}$ with
$\supp(b'_{m'-1})=\{2,k_{m'-1},k_{m'+1}\}$, having the same boundary. Due to the induction hypothesis applied to
$ \sum_{i=j}^{m'-2}\limits \epsilon_i b_i+\epsilon'_{m'-1}b'_{m'-1}$, we are done. When  $k_i\ne k_{j}$ for all $i=j+1,\ldots,m'+1$, we apply the CR-operation to $\epsilon_{j}b_{j}+\epsilon_{j+1}b_{j+1}$, and similarly we are done.
\end{proof}

\begin{remark/def}
In Lemma \ref{section}, we call $\beta''$, a {\em reduct} of $\beta$. The walk sequence of
$\beta''$ is also called a {\em reduct} of the walk sequence of $\beta$.
So given a chain-walk its reducts are also chain-walks, which are obtained
 by the repeated applications of the CR-operation as described  in Lemma \ref{section}.
\end{remark/def}

\begin{theorem}
\label{chain-type}
Let $\alpha$ be a minimal $2$-chain with the
boundary $f_{12} -f_{02} + f_{01}$.
    \be
    \item Assume $\alpha$ is of NR-type. Then  $|\alpha|=1$ or $|\alpha|\geq 5$. If $|\alpha|\geq 5$ then
any chain-walk in $\alpha$ from $f_{01}$ to $-f_{02}$ is of the form
$\sum_{i=0}^{2n}\limits (-1)^i a_i$ which is as a chain equal to $\alpha$  such that $f_{12}=a_{2j}^{1,2}$ for some $1 \le j \le n-1$.
    \item $\alpha$ is of RN-type iff $\alpha$ is equivalent to a $2$-chain
    $$\alpha'=a_0+\sum_{i=1}^{2n-1}\limits \epsilon_i a_i+a_{2n}$$
    $(n\geq 1)$ which is a chain-walk from $f_{01}$ to $-f_{02}$ such that $\Bd^0 a_{2n}=f_{12}$, $\Bd^1(a_{2n})=-f_{02}$ and  $\supp(a_{2n})=\{0,1,2\}$. (The representation of $\alpha'$ is called {\em standard}.)
    \ee
\end{theorem}

\begin{proof}
(1) As mentioned in Remark \ref{existchainwalk},   a chain-walk $\beta$ in $\alpha$ from $f_{01}$ to $-f_{02}$ exists. Now since $\alpha$ is of NR-type,  $\supp(\alpha)=\{0,1,2\}$. If $|\beta|<|\alpha|$ then it follows
$\alpha-\beta$ has  a vanishing support 0, a contradiction. Hence $|\alpha|=|\beta|$ and $\alpha=\beta$.
 Suppose now that  $|\alpha|=3$.  So the chain-walk is $a_0-a_1+a_2=\alpha$ and either $\Bd^0a_0=f_{12}$ or $\Bd^0a_2=f_{12}$. Then
either $\Bd^0 a_1=\Bd^0 a_2$ or  $\Bd^0 a_0=\Bd^0 a_1$. In either case, the subchain of $\alpha$ has a vanishing support 1 or 2, a contradiction. Hence $|\alpha|=1$ or $\geq 5$. When $|\alpha|\geq 5$, all we need to show is that   $f_{12}\ne \Bd^0 a_0$ and $f_{12}\ne \Bd^0 a_{2n}$.  If $f_{12}= \Bd^0 a_0 $
then $\alpha - a_0$ has a vanishing support 1, a contradiction. Hence  $f_{12}\ne \Bd^0 a_0 $. Similarly, we can show $f_{12}\ne \Bd^0 a_{2n} $.

(2)    ($\Leftarrow$) It follows $\supp(\Bd(\alpha'-a_{2n}))= \{0,1\}$, i.e., $\alpha'-a_{2n}$ has a vanishing support, so $\alpha'$ is of RN-type.
 Since CR/RS-operations preserve the minimality and the chain types, $\alpha$ is also an RN-type.
\medskip

($\Rightarrow$)   We prove this in a series of claims. Note that $|\alpha|\geq 3$.
\medskip

  {\em Claim 1.}    There is a 2-chain $\alpha_1\sim \alpha$ which is centered at $2$ such that
$|\supp(\alpha_1)|>3$.
\medskip

{\em Proof of Claim 1.} Let $\alpha_2:=\alpha$ if $|\supp(\alpha)|>3$. Otherwise since $\alpha$ is of RN-type, we can apply RS-operations to obtain some $\alpha_2\sim \alpha$ with $|\supp(\alpha_2)|>3$.
Now there is  $\beta:=\sum_{i\in I}\limits \epsilon_i b_i$, a maximal chain-walk in $\alpha_2$ from $-f_{02}$ to $f_{12}$. If $\beta=\alpha_2$ we put $\alpha_1:=\alpha_2$ and we are done. Otherwise
 let $\gamma:=\alpha_2-\beta$,  and then $\gamma$
 has a vanishing support $2$ in its boundary. By applying the RS-operation to $\gamma$ we find $\gamma'$ with $2\notin\supp(\gamma')$ such that  $\alpha_2\sim \alpha'_2:=\beta+\gamma'$.

 Assume now inductively  we can find a desired $\alpha_1\sim \alpha'_2$ when $|\gamma'|=m$. Let $|\gamma'|=m+1$. Note that $f_{12}-f_{02}+f_{01}\ne \Bd(\beta) $, since otherwise
$\Bd(\gamma')=0$ contradicting the minimality of $\alpha'_2$. Hence there is $i_0\in I$ with $\supp(b_{i_0})=\{2,n_0,n_1\}$
  such that $b_{i_0}^{n_0,n_1}(\ne f_{01})$ with a coefficient, stays in $\Bd(\beta)$.   Therefore there must be a term $\epsilon_{j_0}b_{j_0}$ ($\epsilon_{j_0}\in\{1,-1\})$ in $\gamma'$ such that $(2\notin)\supp(b_{j_0})=\{n_0,n_1,n_2\}$ and $b_{i_0}^{n_0,n_1}=b_{j_0}^{n_0,n_1}$ is cancelled out in $\Bd(\epsilon_{i_0}b_{i_0}+\epsilon_{j_0}b_{j_0})$.
  Now applying the CR-operation to $\epsilon_{i_0}b_{i_0}+\epsilon_{j_0}b_{j_0}$, we get $\epsilon'_{i_0}b'_{i_0}+\epsilon'_{j_0}b'_{j_0}$ with
  $\supp(b'_{i_0})=\{2,n_0, n_2\}, \supp(b'_{j_0})=\{2,n_1, n_2\}$, preserving the boundary. Then from $\beta$, we obtain $\beta'$ by substituting    $\epsilon'_{i_0}b'_{i_0}+\epsilon'_{j_0}b'_{j_0}$ for $\epsilon_{i_0}b_{i_0}$. Notice that $\beta'$ is still a chain-walk from   $-f_{02}$ to $f_{12}$  while $\alpha'_2\sim \alpha''_2:=\beta'+(\gamma'-\epsilon_{j_0}b_{j_0})$. Hence by the induction hypothesis there is a desired $\alpha_1\sim\alpha''_2$. We have proved Claim 1.

\medskip

\begin{figure}[h]
\begin{center}
\scalebox{0.75} 
{
\begin{pspicture}(0,-2.2549999)(11.255,2.2549999)
\rput{-180.0}(10.674374,-0.6949999){\pstriangle[linewidth=0.03,linestyle=dotted,dotsep=0.1cm,dimen=outer](5.337187,-1.8875)(9.7,3.08)}
\psline[linewidth=0.02cm,linestyle=dashed,dash=0.16cm 0.16cm](5.3071876,-1.8475)(1.7071874,1.1725)
\psline[linewidth=0.02cm,linestyle=dashed,dash=0.16cm 0.16cm](5.3271875,-1.8475)(2.9071875,1.1725)
\psline[linewidth=0.02cm,linestyle=dashed,dash=0.16cm 0.16cm](5.3071876,-1.8275)(3.9071877,1.1725)
\psline[linewidth=0.02cm,linestyle=dashed,dash=0.16cm 0.16cm](5.3271875,-1.8075)(4.7471876,1.1725)
\psline[linewidth=0.02cm,linestyle=dashed,dash=0.16cm 0.16cm](5.3071876,-1.8275)(5.7271876,1.1325)
\psline[linewidth=0.02cm,linestyle=dashed,dash=0.16cm 0.16cm](5.3271875,-1.8275)(6.7271876,1.1725)
\psline[linewidth=0.02cm,linestyle=dashed,dash=0.16cm 0.16cm](8.887187,1.1925)(5.3271875,-1.8475)
\psdots[dotsize=0.05](7.087188,0.75249994)
\psdots[dotsize=0.05](7.3871875,0.75249994)
\psdots[dotsize=0.05](7.6671877,0.75249994)
\psdots[dotsize=0.05](7.2671876,1.4125)
\psdots[dotsize=0.05](7.567188,1.4125)
\psdots[dotsize=0.05](7.8471875,1.4125)
\psline[linewidth=0.06cm,arrowsize=0.05291667cm 2.0,arrowlength=1.4,arrowinset=0.4]{cc->}(5.3271875,-1.8675)(0.4871875,1.1925)
\psline[linewidth=0.06cm,arrowsize=0.05291667cm 2.0,arrowlength=1.4,arrowinset=0.4]{cc->}(10.127187,1.1725)(5.3271875,-1.8874999)
\psline[linewidth=0.06cm,arrowsize=0.05291667cm 2.0,arrowlength=1.4,arrowinset=0.4]{cc->}(8.867188,1.1925)(10.187187,1.1925)
\usefont{T1}{ptm}{m}{n}
\rput(5.2814064,-2.0774999){$0$}
\usefont{T1}{ptm}{m}{n}
\rput(0.3996875,1.4325){\small $k_0$}
\usefont{T1}{ptm}{m}{n}
\rput(8.6915,1.4325){\small $k_{2n}$}
\usefont{T1}{ptm}{m}{n}
\rput(10.349688,1.4325){\small $k_{2n+1}$}
\usefont{T1}{ptm}{m}{n}
\rput(1.6796875,1.4325){\small $k_1$}
\usefont{T1}{ptm}{m}{n}
\rput(2.8196876,1.4325){\small $k_2$}
\usefont{T1}{ptm}{m}{n}
\rput(3.8396876,1.4325){\small $k_3$}
\usefont{T1}{ptm}{m}{n}
\rput(4.6796875,1.4325){\small $k_4$}
\usefont{T1}{ptm}{m}{n}
\rput(5.6796875,1.4325){\small $k_5$}
\usefont{T1}{ptm}{m}{n}
\rput(6.6596875,1.4325){\small $k_6$}
\usefont{T1}{ptm}{m}{n}
\rput(1.7114062,0.76250005){$a_0$}
\usefont{T1}{ptm}{m}{n}
\rput(2.7314062,0.76250005){$a_1$}
\usefont{T1}{ptm}{m}{n}
\rput(3.5714061,0.76250005){$a_2$}
\usefont{T1}{ptm}{m}{n}
\rput(4.4514065,0.76250005){$a_3$}
\usefont{T1}{ptm}{m}{n}
\rput(5.1914062,0.76250005){$a_4$}
\usefont{T1}{ptm}{m}{n}
\rput(6.0314064,0.76250005){$a_5$}
\usefont{T1}{ptm}{m}{n}
\rput(9.061406,0.76250005){$a_{2n}$}
\usefont{T1}{ptm}{m}{n}
\rput{-90.0}(-1.3484377,2.1565628){\rput(0.38625,1.7434376){\small $=$}}
\usefont{T1}{ptm}{m}{n}
\rput{-90.0}(6.935562,10.440563){\rput(8.67025,1.743438){\small $=$}}
\usefont{T1}{ptm}{m}{n}
\rput{-90.0}(8.579562,12.084563){\rput(10.314251,1.743438){\small $=$}}
\usefont{T1}{ptm}{m}{n}
\rput(0.4193125,2.0665624){$1$}
\usefont{T1}{ptm}{m}{n}
\rput(8.703313,2.0665624){$1$}
\usefont{T1}{ptm}{m}{n}
\rput(10.347313,2.0665624){$2$}
\usefont{T1}{ptm}{m}{n}
\rput(8.4193125,-0.39749995){$-f_{02}$}
\usefont{T1}{ptm}{m}{n}
\rput(2.5673125,-0.39743745){$f_{01}$}
\usefont{T1}{ptm}{m}{n}
\rput(9.355312,1.4825001){$f_{12}$}
\end{pspicture}  
}
\end{center}
       \caption{A standard RN-type $2$-chain}
\end{figure}
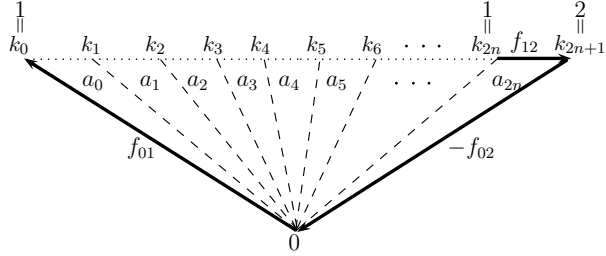

{\em Claim 2.} There is a 2-chain $\alpha_2\sim\alpha_1$ that has a 1-simplex term $c$ (with the coefficient $1$) such that
$\supp(c)=\{0,1,2\}$, and  $f_{12}-f_{02}=
\Bd^0(c)-\Bd^1(c)$.

\medskip

{\em Proof of Claim 2.} For notational simplicity, let $\{0,1,2,3\}\subseteq \supp(\alpha_1)$, and write $\alpha_1=\sum_{i=0}^{2n}\limits\epsilon_i c_i$, a chain-walk from
$-f_{02}$ to $f_{12}$. So for some $j_0\leq 2n$, we have  $\epsilon_{j_0}=1$, $\supp(c_{j_0})=\{0,1,2\}$, and
$\Bd^2(c_{j_0})=f_{01}$.
Let $\beta_0:=\sum_{i=0}^{j_0-1}\limits\epsilon_i c_i$, and $\beta_1:=\sum_{i=j_0+1}^{2n}\limits\epsilon_i c_i$. We shall find the desired $c$ (and $\alpha_2$) by applying the process in Lemma \ref{section} and finding reducts of
chain-walks, starting from $\alpha_1$. Each time, the reduced chain-walk together with the deleted  terms is equivalent to  $\alpha_1$.


\medskip

Case 1) $3\notin \supp(\beta_1)$: So $3\in \supp(\beta_0)$. Now let $I_0:=\langle 0,\ldots,3,\ldots,0\rangle$ be the walk sequence of $\beta_0$, and let $I_1=\langle 1,\ldots,1\rangle$ be the walk sequence of $\beta_1$. So $I_0I_1$ is the walk sequence of $\alpha_1$.  Now  $I_0=J_0J_1$ such that $J_1$ starts with $3$ but all other components $\ne 3$. Then due to Lemma \ref{section} (applied to $J_1I_1$),  we can find  $\gamma_1$, a reduct of  $\alpha_1$, whose walk sequence is $J_0\langle 3,1\rangle.$ Now $J_0=\langle 0,\ldots\rangle$.

If $3$ is not in $J_0$ then again by Lemma \ref{section}, we can further find a reduct of $\gamma_1$  whose walk sequence is $\la 0,3,1\ra$, then again further reduce it with the walk sequence $\la 0,1\ra$, and we are done.

If $3$ is in $J_0$ then in general, by finding a sequence of all $3$'s in $J_0$ and applying Lemma \ref{section}, we can  reduce $J_0$ to a sequence of the form $J'_0=\la 03,k_13,k_23,\ldots;k_{\ell}\ra$ where each $k_i\ne 3$.
If none of the $k_i$'s is $0$ then by applying Lemma \ref{section}
again to $J'_0\la3,1\ra$ we directly reduce it to $\la 0,1\ra$ and we
are done. Otherwise, one of the $k_i$'s is $0$, and we can similarly reduce $J'_0$ to a sequence of the form $\la 03,03,\ldots;k_\ell \ra$. Now the reduced walk sequence is $\la 03,03,\ldots;k_{\ell}; 3,1\ra$.
If $k_{\ell}\ne 1$ then it can directly be reduced to $\la 0,1\ra$ and we are done. If  $k_{\ell}=1$ then
it can be reduced to $\la0,1;3,1\ra$ and further reduced to $\la 0,3,1\ra$ and to $\la 0,1\ra$, so we are done.

 \medskip

   Case 2) $3\notin \supp(\beta_0)$: Then $3\in \supp(\beta_1)$ and the proof will be similar to Case 1.

\medskip

   Case 3) $3\in \supp(\beta_0)\cap \supp(\beta_1)$: By an argument
   similar to that in   Case 1,  the walk sequence of $\alpha_1$ can be in general reduced to $I=\la 03,03,\ldots;k,3\ra\la0,1\ra\la3,k';31,\ldots,31\ra$. Now by the argument in the last part of the proof of Case 1,
  $ \la 03,03,\ldots;k,3\ra\la0,1\ra$ can be reduced to $\la 0,1\ra$. Hence $I$ can be reduced to $\la 0,1\ra\la
  3,k';31,\ldots,31\ra$. Then by the same argument it can finally be reduced to $\la 0,1\ra$, and we have proved Claim 2.

  \medskip

Now lastly we simply take a chain-walk $\gamma$ from $f_{01}$ to $-f_{02}$ in $\alpha_2$ terminating with $c$ ($=$ the $1$-simplex described in Claim 2). Then by an argument similar to that in the proof of Claim 1, we repeatedly  apply the CR-operation to $\gamma$ (while keeping $c$ unchanged), and obtain a desired $\alpha'\sim \alpha_2$ centered at $0$ forming  a chain-walk from $f_{01}$ to $-f_{02}$.
Then we  take the reverse order of the representation of the chain-walk $\alpha'$.
\end{proof}

In an upcoming paper \cite{KL}, it is shown that for any minimal
$2$-chain whose boundary is a $1$-shell, there is an equivalent
$2$-chain which has the same boundary with support size three.


\section{Examples}

 This section is devoted to exhibiting a certain family of examples of $2$-chains of types in rosy theories whose boundaries are 1-shells. The existence of these examples implies that, in rosy theories, there is no uniform bound for the minimal lengths of 2-chains having 1-shell boundaries.

We recall the examples described in \cite{CLPZ}. For a positive
integer $n$, consider a (saturated) structure $M_n=(|M_n|;S,g_n)$, where $|M_n|$ is a circle; $S$ is a ternary relation such that $S(a,b,c)$ holds iff $a,b,c$ are distinct and $b$ comes before $c$ going around the circle clockwise starting at $a$; and $g_n$ is a rotation (clockwise) by $2\pi/n$-radians. 
 When $n$ is obvious from context, $g_n$ is often written  as $g$. The following Fact \ref{basicfact1}, \ref{basicfact2} are from \cite{CLPZ}. 

\begin{fact}\label{basicfact1} 
    \be 
        \item  $\Th(M_n)$  has the unique 1-complete type $p_n(x)$ over $\emptyset$, which is isolated by the formula
        $x=x$.

        \item $\Th(M_n)$ is $\aleph_0$-categorical and has quantifier-elimination.

        \item For any subset $A\subset M_n$, $\acl(A)=\dcl(A)=\bigcup_{0\le
        i<n} g_n^i(A)$ (in the home-sort), where $g_n^i=\underbrace{g_n \circ \cdots \circ g_n}_{i~\text{times}}$.
        \item For each $a\in M_n$ with $n>1$, and an integer $i$, $S(g^i(a),x,g^{i+1}(a))$ isolates a complete type over $a$.
    \ee
\end{fact}

 In what follows, we assume $n>1$.
 


\begin{fact}\label{basicfact2}
    \be
        \item There are $a,b\in M_n$ such that $d(a,b)>n/2$.
        \item For any $a,b\in M_n$, the following are equivalent:
        	\be
        	\item[(i)] $a,b$ begin some $\emptyset$-indiscernible sequence,
        	\item[(ii)] $a$ and $b$ have the same type over some elementary substructure of $M_n$,
        	\item[(iii)]  $a=b\vee S(a,b,g_n(a))\vee S(b,a,g_n(b))$ holds.
        	\ee
\ee
\end{fact}

Thus the unique 1-complete type $p_n$ is also a Lascar type.

\begin{theorem}\label{weak} 
\be\item
$\Th(M_n)$ has weak elimination of
    imaginaries.
\item    
   $\Th(M_n)$ is rosy having thorn $U$-rank $1$ with a trivial pregeometry. 
    \ee
\end{theorem}
\begin{proof} (1) We claim that if a   set $D$ in $(M_n)^k$ is definable over $A_0$ and $A_1$ respectively where $A_i=\acl(A_i)=\dcl(A_i)$ (in the home-sort) then it is definable over $B:=A_0\cap A_1$:   We sketch the proof of the claim by freely using Fact \ref{basicfact1}. Let $k=1$.
  Due to quantifier elimination, $D$ is some union of finitely many arcs on $M_n$.
    Clearly each end-point of a connected component of $D$ is in $\dcl(A_i)$ and so in $B$ as well. Hence $D$ is indeed  $B$-definable. Now for induction, assume the claim holds for $k-1$. We want to show it holds for $k$.    
   Suppose that $\varphi_i(x_1,\ldots,x_k,\bar a_i)$ defines
   $D$ where $\bar a_i\in A_i$. Then, for each element $b$, the set $D_b$ defined by $\varphi_i(x_1,\ldots,x_{k-1},b,\bar a_i)$ is definable over $Bb$, by the induction hypothesis. But due to $\aleph_0$-categoricity (so there are only finitely many formulas over $\emptyset$ up to equivalence), it easily follows that for each $y$, 
   $\varphi_i(x_1,\ldots,x_{k-1},y,\bar a_i)$ is definable over $B$, i.e. $D$ is definable over $B$ as we wanted. 
  
\medskip

Now let $E(\bar x,\bar y)$ be an $\emptyset$-definable equivalence relation on $(M_n)^k$.  For
$\bar a\in (M_n)^k$, let $\bar a'$ denote a finite tuple of algebraic closure of $\bar a$ in the home-sort. Let $\bar b$ be the maximal subtuple of  $ \bar{a}'$ which is algebraic over $\bar a/E$. Thus there is $\bar a''\equiv_{\acl(\bar a/E)}\bar a'$ such that $\bar b=\bar a'\cap \bar a''$ as sets. Hence due to the claim,
$\bar a/E\in \dcl^{\eq}(\bar b)$ and $\bar b\in \acl(\bar a/E)$. We have proved (1). 

    
    (2) Due to (1), $\Th(M_n)$ is rosy having thorn $U$-rank $1$ as pointed out in \cite{EO}.
Notice that $M_n$ has the same pregeometry as the  $n$-copies of a
half-closed interval, and so $M_n$ forms a trivial pregeometry with
its algebraic closures.
\end{proof}


\begin{definition}
Let $a,b\in M_n$ be any elements with $\acl(a)\neq\acl(b)$.
\begin{enumerate}
\item We define the
{\em $S$-distance of $b$ from $a$}, denoted by $\sdist(a,b)$ as
follows: $\sdist(a,b)=k$ iff $M_n \models
S(g^k(a),b,g^{k+1}(a))$. For integers $k<l$, we write
$k\le\sdist(a,b)\le l$ if $M_n \models S(g^k(a),b,g^{l+1}(a))$.
\item We define the {\em $\widehat{S}$-distance of $b$ from $a$},
  denoted by $\bsdist(a,b)$, as similar manner as $\sdist(a,b)$, using
  the formula $$\widehat{S}(x,y,z)\equiv (x\neq z \wedge S(x,y,z))
  \vee (x=z\wedge x\neq y).$$ 
\end{enumerate}
\end{definition}

\begin{remark} \label{trianglebridge}
Let  $x, y, z\in M_n$ have mutually disjoint algebraically
closures. Then for any  $k, l, m\in \mathbb{Z}$,
\be 
 \item[$(1)$] $\sdist(y,x)=-\sdist(x,y)-1$;
 \item[$(2)$] 
    \be
            \item for $l-k\not\equiv -1,0 \pmod n$, if $k\le\sdist(x,y)\le l-1$, and $\sdist(y,z)=m$, then $m+k\le \sdist(x,z)\le m+l$;
            \item for $l-k\equiv -1 \pmod n$, if $k\le\sdist(x,y)\le l-1$, and $\sdist(y,z)=m$, then $g^{k+m}(x)\neq z$.
        \ee

   \item[$(1)'$] $\bsdist(y,x)=-\bsdist(x,y)-1$;

   \item[$(2)'$]  for $k\not\equiv l \pmod n$, if $k\le\bsdist(x,y)\le
      l-1$, and $\bsdist(y,z)=m$, then $m+k\le \bsdist(x,z)\le
      m+l$.
\ee

\end{remark}

\begin{lemma}\label{trianglebridget}
Let $k $ and $l_0, \ldots, l_m$ be fixed integers and $L_j :=
\sum_{i=0}^j l_i$. Let $a$ and $d_0,\ldots,d_{m+1}$ $(m+1<n)$ be elements in $M_n$ such that
$$(*)_{m} : \bsdist(a,d_0)=k,\ \bsdist(d_i,d_{i+1})=l_i, \ 0\leq i \leq m.$$
Then $$k+L_m \le \bsdist(a,d_{m+1})\le k+L_m +m+1.$$
Moreover, by choosing appropriate elements for $a$ and
$d_0,\ldots,d_{m+1}$, the quantity $\bsdist(a,d_{m+1})$ can be made to
be any integer in $[\:k+L_m,\:k+L_m+m+1\:]$ $(**)_m$.
\end{lemma}

\begin{proof}
We show this using induction on $m$.
For $m=0$, by Remark \ref{trianglebridge}$(2)'$, it follows from $(*)_0$ that
$$k+l_0 \le\bsdist(a,d_1)\le k+l_0+1.$$
Moreover it is not hard to see $(**)_0$ holds.

Now assume the lemma holds for $m-1$ with $m+1<n$. Let us show the
lemma for $m$. For $i\le m+1$, $a, d_i \in M_n$ are given which satisfy $(*)_m$. 
 Firstly, by the induction hypothesis for $m-1$,
$$k+L_{m-1}\le\bsdist(a,d_m)\le k+L_{m-1}+m.$$
Since $m+1<n$,
 $$k+L_{m-1} \le \sdist(a,d_m) \le k+ L_{m-1} +m.$$
Then again by Remark \ref{trianglebridge}$(2)'$,
$$k+L_m\le \bsdist(a,d_{m+1})\le k+L_m+m+1.$$
Secondly, we show the moreover part. 
Fix $L_m \le j \le L_m+m+1$ and $a'\in M_n$. If $j=L_m$, then $j-l_m=L_{m-1}$ and due to the induction hypothesis, there are $d'_0,\ldots,d'_m$ that satisfy $(*)_{m-1}$ and
$$\bsdist(a',d'_m)=k+j-l_m.$$ 
So, $\bsdist(a',g^{l_m}(d'_m))=k+j$, and $M_n
\models  S\:(g^{l_m}(d'_{m}),d'_{m+1},g^{k+j+1}(a'))$ for some $d'_{m+1} \in M_n$.
Thus $$\bsdist(d'_m,d'_{m+1})=l_m,\ \bsdist(a',d'_{m+1})=k+j.$$ So, $a'$ and $d_i'$ for $i\le m+1$ satisfy the required condition. Now for $j>L_m$, the proof is similar to the case $j=L_m$ except that
we replace $j-l_m$ by $j-l_m-1$ and take $d'_{m+1}$ in $M_n$ such that
$$M_n \models S(g^{k+j}(a'),d'_{m+1},g^{l_m+1}(d'_{m})).$$
\end{proof}

Now, let $\CA(p_n)$ be the family of all the closed independent
functors in $p_n$. We follow the notation given at the
beginning of Section \ref{sec:h1}:  given a closed independent functor $f$ over $\emptyset$ in $p_n$ with $u=\{i_0<\cdots <i_k\}\in \dom(f)$, we write $f(u)=[a_0,\ldots,a_k]$, where $a_j \in M_n$, $f(u)=\acl(a_0,\ldots ,a_k)$, and $\acl(a_j)=f^{\{i_j\}}_u(\{i_j\})$. When we write $f(u)\equiv [b_0,\ldots,b_k]$, it of course means that $[a_0,\ldots,a_k]\equiv [b_0,\ldots,b_k]$. By Theorem \ref{weak}, it is equivalent to saying $a_0\cdots a_k\equiv b_0\cdots b_k$.

\begin{remark}\label{concatenatetriangles}
Let $\tau=\sum_{i=0}^m\limits \epsilon_i t_i$ ($t_i$ 2-simplex) be a chain-walk (in $p_n$) from $f_{01}$ to $-f_{02}$ such that $D_i=\supp(t_i)=\{0,k_i,k_{i+1}\}$ with $k_0=1$, $k_{m+1}=2$. Then putting together the {\em triangles} $t_0(D_0),\ldots,$ $t_m(D_m)$ side by side {\em centered at} $0$, we can find elements $a$ and $d_0,\ldots,d_{m+1}$ in $M_n$ such that for $0\le i\le m$,
    $$t_i(D_i)\equiv \begin{cases}
                    [a,d_i, d_{i+1}] & \text{if}~ k_i<k_{i+1}\\
                    [a,d_{i+1}, d_i] & \text{if}~ k_i>k_{i+1}.
                    \end{cases}$$
\end{remark}







Combining the classification results in Section \ref{sec:clsf} and Lemma \ref{trianglebridget}, we will show
 that there does not exist any finite upper bound for the minimal lengths of 2-chains with 1-shell boundaries in the types $p_n$.


%
%

\begin{theorem}\label{nobound}
Let $\CA$ be a non-trivial amenable collection and let $s$ be a 1-shell. Define $B(s)$, and $B(A)$ as follows:
\be
    \item $B(s) :=\min \{~|\tau|  :  \tau~\text{is a (minimal) 2-chain and}~\partial(\tau)=s ~\}$.
    
\noindent    (If $s$ is not the boundary of any $2$-chain,  define $B(s) := - \infty$.)
    
    \smallskip

    \item $B(\CA):=\max \{ B(s)\colon s~\text{is a 1-shell of}~\CA ~\}$.
\ee
Let $n>1$ and let $s=s_{12}-s_{02}+s_{01}$ be a 1-shell from $\CA(p_n)$ with $\supp(s_{ij})=\{i,j\}$. Then there are $a,b,c,c'$ in $M_n$ and some integers $k_1,k_2,k_3$ with $0\le k_i <n$  such that,
\be
	\item[$\bullet$] $\bsdist(a,c)=k_1$, $\bsdist(a,b)=k_2$, and $\bsdist(b,c')=k_3$;  
	
	\item[$\bullet$]  $s_{01}(\{0,1\})\equiv [a,c],s_{02}(\{0,2\})\equiv [a,b],$ and $s_{12}(\{1,2\})\equiv [c',b]$.
\ee
Let $0\leq k_4(< n)\equiv k_2-(k_1-k_3)$ (mod $n$) and let $$n_s:=\min\{ 2(n-k_4)-1,\; 2k_4+1\}.$$
Then $$B(s)=n_s.$$ Moreover, taking $k_1=0,$ $k_2=0$, and $k_3=[\frac{n}{2}]$, we get $n_s\ge n-1$ and $B(\CA(p_n))\ge n-1$. Therefore $\lim_{n\rightarrow\infty}\limits B(\CA(p_n))=\infty.$
\end{theorem}
\begin{proof}
(1) $B(s)\ge n_s$ : By Theorem \ref{h1=0} and Corollary
\ref{chain-type}, there is a chain-walk $\tau=\sum_{i=0}^{2m}\limits
(-1)^i t_i$ from $s_{01}$ to $-s_{02}$
and $\Bd(\tau)=s$. We want to show $|\tau|\ge n_s$. Suppose not,  i.e., $|\tau|=2m+1 <n-1$. By Remark \ref{concatenatetriangles}, there are $d_i$'s ($0\le i\le 2m+1$) in $M_n$ such that $ac\equiv ad_0,~d_{2m+1}=b$; 
and
\be
    \item[$\bullet$] $\bsdist(d_0,d_1)=l_0,\ \bsdist(d_{2m-1},d_{2m})=l_{2m}$ for some integers
    $l_i$;

    \item[$\bullet$] $t_0(\{0,k_0,k_1\})\equiv [a,d_0,d_1]$, $t_{2j-1}(\{0,k_{2j-1},k_{2j}\})\equiv [a,d_{2j},d_{2j-1}]$, and $t_{2j}(\{0,k_{2j},k_{2j+1}\})\equiv[a,d_{2j},d_{2j+1}]$ for $1\le j\le m$. 
\ee
Now $\Bd\tau=s$ implies $\Bd^0 t_{2j_0}=s_{12}$ for some $0 \le 2j_0\le 2m$; and for any $0\le j_1\neq j_0 \le m$ there is $0\le j_2\neq j_0 \le m$ (indeed a bijection) such that $\Bd^0 t_{2j_1}=\Bd^0 t_{2j_2+1}$. So
\be
	\item[$\bullet$] $\bsdist(d_{2j_0},d_{2j_0 +1})=-k_3-1$; and
	
	\item[$\bullet$] $[d_{2j_1},d_{2j_1+1}]\equiv [d_{2j_2+2},d_{2j_2+1}]$.
\ee
By Remark \ref{trianglebridge}$(1)'$, $\bsdist(d_{2j_1},d_{2j_1 +1})=-\bsdist(d_{2j_2},d_{2j_2 +1})-1$. Therefore $l_{2j_0}=-k_3-1$ and $l_{2j_2+1}=-l_{2j_1}-1$, so $\sum_{j=0}^{2m}\limits l_j=-k_3-m-1$. Hence due to Lemma \ref{trianglebridget} and $2m+1<n-1$, we have $k_1-k_3-m-1\le \bsdist(a,b)\le k_1-k_3+m$. Thus
$$\bsdist(a,b)=k_2;\mbox{ and}\ k_1-k_3-m-1\le \bsdist(a,b)\le k_1-k_3+m.$$
We rewrite it as
$$\bsdist(g^{k_1-k_3}(a),b)=k_2-(k_1-k_3);\mbox{ and}\ -m-1\le \bsdist(g^{k_1-k_3}(a),b)\le m.$$
We can replace $k_2-(k_1-k_3)$ by $k_4$ and we have $n-(m+1)<k_4+1$ or $m+1>k_4$. In either case, we have $m\ge \min \{n-k_4-1,\; k_4\}$. Therefore $2m+1\ge 2\min \{n-k_4-1,\; k_4\} +1=\min \{2(n-k_4)-1,\; 2k_4+1 \}=n_s$, a contradiction. We have proved  $B(s)\ge n_s$.

(2) $B(s)\le n_s$ : We construct a chain-walk 
$\gamma=\sum_{i=0}^{n_s-1}\limits r_i$ with $\supp(\gamma)=\{0,1,2\}$ and $\Bd \gamma=s$ as follows: Note that
since $n_s$ is odd, $m_s:=(n_s-1)/{2}$ is an integer. Also note that,
if we let $N_1:=k_1-k_3-m_s-1$ and $N_2:=k_1-k_3+m_s$, then  $k_2
\equiv N_i \pmod n \ (i=1 \mbox{ or }2).$ Hence we have $ \bsdist(a,b)=N_1$ or $\bsdist(a,b)=N_2$.
Applying Lemma \ref{trianglebridget} with $k_1$ and $l_0,\ldots,l_{2m_s}$ such that $l_{2i+1}=-l_{2i}-1$
for $0\leq  i< m_s$ and $l_{2m_s}=-k_3-1$, we obtain $\sum_{i=0}^{2m_s}\limits l_i=L_{2m_s}= -m_s-k_3-1$, and 
$L_{2m_s} +2m_s+1=m_s-k_3$. Therefore if $j$ is chosen to be such that
$j=N_1-k_1$ or $=N_2-k_1$, and by applying $(**)_{2m_s}$ in Lemma \ref{trianglebridget}, we can
find  $d'_0,\ldots,d'_{2m_s+1}(=d'_{n_s})$ such that 
$$\bsdist(a,d'_0)=k_1,\ \bsdist(d'_i,d'_{i+1})=l_i\mbox{ for } 0\le i\le 2m_s,\ \mbox{and } \bsdist(a,d'_{n_s})=k_2.$$ 
Then due to Fact \ref{basicfact1}(4) and Remark \ref{trianglebridge}
$(1)'$, it follows that 
$ad'_0\equiv ac,~ad'_{n_s}\equiv ab,~d'_{n_s-1}d'_{n_s}\equiv c'b$ and $d'_{2i}d'_{2i+1}\equiv
d'_{2i+1}d'_{2i+2}$ for $0\leq i< m_s$.
Hence clearly we have a desired 2-chain $\gamma=\sum_{i=0}^{n_s-1}\limits r_i$ such that 
$$r_i(\{0,1,2\})\equiv
	\begin{cases}
	[a,d'_i,d'_{i+1}] & \text{if}~i\equiv 0\pmod 2\\
	[a,d'_{i+1},d'_i] & \text{if}~i\equiv 1\pmod 2.
	\end{cases}$$
\end{proof}

\begin{corollary}\label{no3weak}
For each $n\ge 5$, $\CA(p_n)$ does not have weak 3-amalgamation.
\end{corollary}

\end{document}